\documentclass[12pt,a4paper]{amsart}
\usepackage{amssymb,amsmath,amsthm,verbatim}

\theoremstyle{plain}

\newtheorem{theorem}{Theorem}

\theoremstyle{remark}
\newtheorem*{remark}{Remarks}

\begin{document}

\title[Quantum stochastic integrals]{Quantum stochastic integrals
as operators}
\author{Andrzej \L uczak}
\address{Faculty of Mathematics and Computer Science, \L \'od\'z
University, \ ul. Banacha 22,
90--238 \L\'od\'z, Poland}
\email{anluczak@math.uni.lodz.pl}
\thanks{Work supported by KBN grant 2 P03A 03024}

\begin{abstract}
We construct  quantum stochastic integrals for the
integrator being a martingale in  a von Neumann algebra, and the
integrand --- a suitable process with values in the same algebra, as
densely defined operators affiliated with the algebra. In the case
of a finite algebra we allow the integrator to be an
$L^2$--martingale in which case the integrals are $L^2$--martingales
too.
\end{abstract}
\maketitle

\section*{Introduction} The theory of `general' quantum stochastic
integrals (i.e. not founded on Fock space) deals mainly with the
setup which can be roughly described as follows: For a  von Neumann
algebra $\mathcal{A}$ with a filtration
$\{\mathcal{A}_t:t\geqslant0\}$ we have a corresponding process
$(X(t): t \geqslant 0)$ with values in $L^p(\mathcal{A})$, and a
corresponding process \ $(f(t): t \geqslant 0)$ \ with values in \
$L^q(\mathcal{A})$, $1/p + 1/q = 1/2$, $2 \leqslant p$,  $q \leqslant
+ \infty$, $L^p(\mathcal{A})$ and $L^q(\mathcal{A})$ being
appropriate noncommutative $L^p$-spaces. Then under various specific
assumptions about $X$ and $f$, among which the most natural is that
$X$ is a martingale, one can define stochastic integrals $\int^b_a f
\, dX $ and  $\int^b_a dX \, f$ as elements of $L^2(\mathcal{A})$
(cf. \cite{1,4,5,7,8,9,10}).  The advantage of this is that, for any
reasonable definition of the integral, it may be approximated by
integral sums of the form $\sum f(t_{i-1})[X (t_{i}) - X(t_{i-1})]$
(for $\int^b_a f \, dX $, the other one being defined in complete
analogy). Now a product of elements from $L^p(\mathcal{A})$ and
$L^q(\mathcal{A})$, is an element from $L^r(\mathcal{A})$, where $1/p
+ 1/q = 1/r$. In particular, the approximation holds in the case
$r=2$ (in a Hilbert space). However, this approach can be exploited
also in the following quite natural setting. If we assume that the
algebra $\mathcal{A}$ acts in a Hilbert space $\mathcal{H}$, and if
we let $f$ and $X$ take their values in $\mathcal{A}$, then the
integral sum belongs to $\mathcal{A}$ too, and we may ask about its
behavior on elements  of $\mathcal{H}$. This again leads us to
approximation of the integral evaluated at some points of
$\mathcal{H}$, i.e., we come to the notion of integrals $\int^b_a f
\, dX $ and  $\int^b_a dX \, f$ as operators on $\mathcal{H}$. This
idea has been carried out in \cite{3} for a particular class of
martingales $X$ defined `canonically' in quasi-free representations
of the CCR and CAR.

It turns out that it is possible to apply this point of view in the
following situation: As the integrator we take a monotone or
norm-continuous $\mathcal{A}$-valued process. Then we show that there
exists a Riemann-Stieltjes type integral $\int_a^b f\,dX$ as a
densely defined operator affiliated with the algebra $\mathcal{A}$.
In the case when $\mathcal{A}$ is finite, we show the existence of
both the integrals $\int_a^b f\,dX \text{ and }\int_a^b dX\,f$ as
densely defined closed operators affiliated with $\mathcal{A}$.
Moreover, we can weaken the assumptions about $(X(t))$ and allow it
to be a martingale in $L^2(\mathcal{A})$, in which case the integrals
above will be elements from $L^2(\mathcal{A})$ too.

Finally, let us say a few words about quantum stochastic
integrals in general. In the existing theories of quantum
integration, especially those of Lebesgue type, the classes of
`theoretically admissible' integrands are rather narrow and lack
any \emph{concrete} examples of processes that can be integrated.
On the other hand for integrals of Riemann-Stieltjes type such
examples have been provided, and, for example, it was shown how
one can integrate predictable processes and how integration with
respect to a quantum random time can be performed (cf.
\cite{8,9}). In this note we take a similar approach showing the
possibility of integrating either a monotone or norm-continuous
process with respect to a martingale. It is worth noting that in
the case of a monotone process, apparently nothing can be said
about the existence of Lebesgue type integrals, while in the case
of a norm-continuous process its existence can be proved only
under some additional assumption.

\section{Preliminaries and notation} Throughout the paper, we assume
that $\mathcal{A}$ is a von  Neumann algebra acting in a Hilbert
space $\mathcal{H}$ with inner product $\langle \cdot, \cdot
\rangle$, and that $\omega$ is a normal faithful state given by a
cyclic and separating unit vector $\varOmega$ in  $\mathcal{H}$,
i.e.,  $\mathcal{A}$ is represented in standard form. We suppose,
further,  that  we have an increasing family $\{\mathcal{A}_t:
t\in[0,+\infty)\}$ of von Neumann subalgebras of $\mathcal{A}$
($\mathcal{A}_s\subset\mathcal{A}_t$ for $s\leqslant t$), called a
filtration, and a corresponding family $\{ \mathbb{E}_t\}$ of normal
conditional expectations from $\mathcal{A}$ onto $\mathcal{A}_t$
leaving $\omega$ invariant.

A \emph{process} in $\mathcal{A}$ or $L^2(\mathcal{A})$ is a function
defined on $[0,+\infty)$ with values in $\mathcal{A}$ or
$L^2(\mathcal{A})$, respectively. We shall denote by $f$, processes
in $\mathcal{A}$, and by $X$, processes either in $\mathcal{A}$ or
$L^2(\mathcal{A})$. Following the notation of probability theory, we
shall sometimes denote a process by \linebreak $(X(t):t \geqslant 0)$
(a family of `random variables'), and the same applies to $f$.

The norms in $\mathcal{H}\text{ and }L^2(\mathcal{A})$ will be
denoted $\|\cdot\|_{\mathcal{H}}\text{ and }\|\cdot\|_2$,
respectively, while $\|\cdot\|$ will stand for the operator norm.

Define on (a dense subspace of) $\mathcal{H}$ operators $P_t$ given
by:
\begin{equation}\label{1}
P_t(x\varOmega) = (\mathbb{E}_tx)\varOmega,\qquad x\in\mathcal{A}.
\end{equation}
It is well--known (cf. e.g. \cite{6} Propositions 1.1 and 1.2) that\\
$(P_t:t\geqslant0)$ is an increasing family of orthogonal projections
in $\mathcal{H}$  with ranges $\mathcal{H}_t =
\overline{\mathcal{A}_t \varOmega}$. Moreover,
$P_t\in\mathcal{A}_t'$, where $\mathcal{A}_t$ is the commutant of
$\mathcal{A}_t$.

A process in  $\mathcal{A}$ or $L^2(\mathcal{A})$ will be called
\emph{adapted} if its value at each point $t\geqslant0$ belongs to
$\mathcal{A}_t$ or $L^2(\mathcal{A}_t)$, respectively. We call
processes $f$ in  $\mathcal{A}$, and $X$ in $L^2(\mathcal{A})$
\emph{martingales} if for any $s,t \in [0,+\infty)$,
 $s \leqslant t$, the equalities
\[
 \mathbb{E}_s f(t) = f(s), \qquad \mathbb{E}_s X(t) = X(s)
\]
hold, where in the $L^2$-case we use the same symbol $\mathbb{E}_s$
to denote the extension of the conditional expectation to
$L^2(\mathcal{A})$. It follows that a martingale is an adapted
process.

Let $(X(t):t\in[0,+\infty))$ be a process, and let $0\leqslant
t_0\leqslant t_1\leqslant\dots\leqslant t_m<+\infty$ be a sequence of
points. To simplify the notation we put
\[
 \Delta X(t_k)=X(t_k)-X(t_{k-1}),\qquad k=1,\dots,m.
\]

Let $(X(t):t\in[0,+\infty)),\,(f(t):t\in[0,+\infty))$ be arbitrary
processes, and let $[a,b]$ be a subinterval of $[0,+\infty)$. For a
partition \linebreak $\theta=\{a=t_0<t_1<\dots<t_m=b\}$ of $[a,b]$ we
form left and right integral sums
\[
 S_{\theta}^l=\sum_{k=1}^m\Delta X(t_k)f(t_{k-1}),\qquad
 S_{\theta}^r=\sum_{k=1}^m f(t_{k-1})\Delta X(t_k).
\]
If there exist limits (in any sense) of the above sums as $\theta$ is
refined, we call them respectively the \emph{left} and \emph{right
stochastic integrals} of $f$ with respect to $(X(t))$, and denote
\[
 \lim_{\theta}S_{\theta}^l=\int_a^b dX\,f,\qquad
 \lim_{\theta}S_{\theta}^r=\int_a^b f\,dX.
\]
This notion of integral is a weaker one than defining
the integrals as the limits
\[
 \int_a^b dX\,f=\lim_{\|\theta\|\to 0}S_{\theta}^l,\qquad
 \int_a^b f\,dX=\lim_{\|\theta\|\to 0}S_{\theta}^r,
\]
where $\|\theta\|$ stands for the mesh of the partition $\theta$. A
definition of this kind is standard in the classical theory of
Riemann-Stieltjes as well as the theory of stochastic integrals. It
is worth noticing that in noncommutative integration theory, whenever
this Riemann-Stieltjes type integral is considered, its definition
refers to the weaker form of the limit with the refining net of
partitions (cf. \cite{2,8,9}). However, under additional assumptions
we shall be able to obtain the integral also in the stronger sense
thus making it similar to the classical stochastic integral.

\section{Integrals as operators on $\mathcal{H}$ --- the non-tracial case}
Our construction of the integral is given by the following

\begin{theorem}\label{mi} Let $(X(t): t\geqslant 0)$ be a martingale
in $\mathcal{A}$, and let\linebreak $f\colon
[0,\infty)\to\mathcal{A}$ be a hermitian adapted process such that
$f$ is monotone. Then for each $t>0$ there exists a
Riemann--Stieltjes type integral $\int^t_0 f\,dX$, which is a
densely defined operator on $\mathcal{H}$ affiliated with
$\mathcal{A}$.
\end{theorem}
\begin{proof}Fix $t>0$,and let $\theta = \{0 = t_0 < \ldots < t_m = t\}$
be a partition of $[0,t]$. Put
\[
 S^r_{\theta}= \sum^m_{k=1} f(t_{k-1})[X (t_k)-X(t_{k-1})].
\]
We want to define $\int^t_0 f \, dX$ as an operator on
$\mathcal{A}'\varOmega$ (where $\mathcal{A}'$ is the commutant of
$\mathcal{A}$) by
\[
 \int^t_0 f\, dX (x'\varOmega) =  \lim_{\theta }
 S^r_{\theta}(x'\varOmega),\qquad x'\in\mathcal{A}',
\]
as $\theta$ is refined. For this it is sufficient to show the
existence of the limit $\lim_{\theta}S^r_{\theta}\varOmega$, since
\[
  S^r_{\theta}\, x' = x'\,S^r_{\theta}.
\]
We have
\[
 \begin{split}
 S^r_{\theta}\,\varOmega &=\sum^m_{k=1}f(t_{k-1})[X (t_k) - X(t_{k-1})]\,\varOmega\\
 &=\sum^m_{k=1}f(t_{k-1})[\mathbb{E}_{t_k} X(t)-\mathbb{E}_{t_{k-1}}X(t)]\,\varOmega\\
 &=\sum^m_{k=1}f(t_{k-1})(P_{t_k} - P_{t_{k-1}}) \,X(t)\varOmega,
 \end{split}
\]
where the $P_t$ are projections defined by \eqref{1}. Consider
the operator sum
\begin{equation}\label{2}
 \sigma^r_{\theta}=\sum^m_{k=1}f(t_{k-1})(P_{t_k}-P_{t_{k-1}})=
 \sum^m_{k=1}(P_{t_k}-P_{t_{k-1}})\,f(t_{k-1})=(\sigma^r_{\theta})^*.
\end{equation}
To fix attention, assume that $f$ is increasing (i.e., $f(s)\leqslant
f(t)$ for $s\leqslant t$), and let $\theta'=\theta\cup\{t'\}$ for
some $t_j<t'< t_{j+1}$ be a one-point refinement of $\theta$. Then
\[
\begin{split}
\sigma^r_{\theta'} - \sigma^r_{\theta} & = f(t_j) (P_{t'} - P_{t_j})
+f(t') (P_{t_{j+1}} - P_{t'})
- f(t_j) (P_{t_{j+1}} - P_{t_j})\\
& = [f(t') -  f(t_j)]  (P_{t_{j+1}} - P_{t'}) \\
& =  (P_{t_{j+1}} - P_{t'})   [f(t') -  f(t_j)]  (P_{t_{j+1}} -
P_{t'}) \geqslant 0,
\end{split}
\]
because $P_{t_{j+1}} - P_{t'}$ is a projection commuting with
$f(t_j)$ and  $f(t')$, and  $f(t')-f(t_j)\geqslant0$. It follows
that the net $\{\sigma^r_{\theta}\}$ is increasing. Furthermore,
for each $k=1,\ldots,m$ we have
\[
 (P_{t_k}-P_{t_{k-1}})\,f(t_{k-1})^2(P_{t_k}-P_{t_{k-1}})
 \leqslant\| f(t_{k-1})\|^2(P_{t_k}-P_{t_{k-1}}).
\]
Hence
\[
 \begin{split}
 \sum^m_{k=1}&(P_{t_k}-P_{t_{k-1}})\,f(t_{k-1})^2(P_{t_k}-
 P_{t_{k-1}})\leqslant\sum^m_{k=1}\| f(t_{k-1})\|^2(P_{t_k}-P_{t_{k-1}})\\
 &\leqslant c^2 \sum^m_{k=1}(P_{t_k}-P_{t_{k-1}})= c^2(P_t-P_0),
 \end{split}
\]
where $c=\sup_{0\leqslant s\leqslant t}
\|f(s)\|=\max\{\|f(0)\|,\|f(t)\|\}$. Consequently,
\[
 \begin{split}
 \|\sigma^r_{\theta}\|^2 & = \|( \sigma^r_{\theta})^2\|=
 \Big\|\sum^m_{i,j=1}(P_{t_j}-P_{t_{j-1}})f(t_{j-1}) f(t_{i-1})
 (P_{t_i} - P_{{t_{i-1}}})\Big\| \\
 &=\Big\|\sum^m_{i=1}(P_{t_i}-P_{t_{i-1}})f(t_{i-1})^2(P_{t_i}-P_{t_{i-1}})
 \Big\|\leqslant c^2\|P_t-P_0\|=c^2,
 \end{split}
\]
which means that $\{\sigma^r_{\theta}\}$ is norm-bounded. This,
together with the fact that the net is increasing, yields the
existence of $\lim_{\theta}\sigma^r_{\theta}$ in the strong
operator topology, in particular, there exists
\[
 \lim_{\theta}\sigma^r_{\theta}(X(t)\varOmega)
 = \lim_{\theta}S^r_{\theta}\varOmega.
\]
It is clear that for each $x',y'\in\mathcal{A}'$ we have
\begin{align*}
 &\bigg(\int_0^t
 f\,dX\bigg)y'(x'\varOmega)=\lim_{\theta}S_{\theta}^ry'(x'\varOmega)=\\
 &y'\lim_{\theta}S_{\theta}^r(x'\varOmega)=y'\bigg(\int_0^t
 f\,dX\bigg)(x'\varOmega),
\end{align*}
thus $\int_0^t f\,dX $ is affiliated with $\mathcal{A}''=\mathcal{A}$.
\end{proof}
It turns out that even a stronger form of integral can be obtained
for norm-continuous processes.
\begin{theorem}\label{ci}
Let $(X(t):t\geqslant 0)$ be a martingale in $\mathcal{A}$, and
let\linebreak $f\colon[0,\infty)\to\mathcal{A}$ be a
norm-continuous adapted process. Then for each $t>0$ there exists a
Riemann-Stieltjes type integral $\int_0^t f\,dX$ which is a
densely defined operator on $\mathcal{H}$ affiliated with
$\mathcal{A}$. Moreover, this integral is given as the limit
\[
 \lim_{\|\theta\|\to 0}\sum_{k=1}^m f(t_{k-1})\Delta X(t_k)
\]
on $\mathcal{A}'\varOmega$.
\end{theorem}
\begin{proof}
Fix $t>0$. We shall show that the net $\{S_{\theta}^r\varOmega\}$
is Cauchy as the mesh $\|\theta\|$ of the partition $\theta$ tends
to 0. Take an arbitrary $\varepsilon>0$, and let $\delta>0$ be
such that for each $t',t''\in[0,t]\text{ with }|t'-t''|<\delta$ we
have
\[
 \|f(t')-f(t'')\|<\frac{\varepsilon}{2\|X(t)\varOmega\|_{\mathcal{H}}}.
\]
Let $\theta'=\{0=t_0<t_1<\dots<t_m=t\}$ be an arbitrary partition
of $[0,t]$ with $\|\theta'\|<\delta$, and let $\theta''$ be a
partition of $[0,t]$ finer than $\theta'$. Denote by $t_0^{(k)},
t_1^{(k)},\dots,t_{l_k}^{(k)}$ the points of $\theta''$ lying
between $t_{k-1}\text{ and }t_k$, such that
$t_{k-1}=t_0^{(k)}<t_1^{(k)}<\dots<t_{l_k}^{(k)}=t_k$. We then
have
\begin{align*}
 S_{\theta''}^r&=\sum_{k=1}^m\sum_{i=1}^{l_k}f(t_{i-1}^{(k)})
 \Delta X(t_i^{(k)})
 \\
 S_{\theta'}^r&=\sum_{k=1}^mf(t_{k-1})\Delta X(t_k)=
 \sum_{k=1}^m\sum_{i=1}^{l_k}f(t_{k-1})\Delta X(t_i^{(k)}),
\end{align*}
so that
\[
 S_{\theta''}^r-S_{\theta'}^r=\sum_{k=1}^m\sum_{i=1}^{l_k}
 [f(t_{i-1}^{(k)})-f(t_{k-1})]\Delta X(t_i^{(k)}).
\]
As in the proof of Theorem \ref{mi} we have
\begin{align*}
 (S_{\theta''}^r-S_{\theta'}^r)\varOmega=&\sum_{k=1}^m
 \sum_{i=1}^{l_k}\left[f(t_{i-1}^{(k)})-f(t_{k-1})\right]
 \Delta X(t_i^{(k)})\varOmega\\
 =&\sum_{k=1}^m\sum_{i=1}^{l_k}\left[f(t_{i-1}^{(k)})-f(t_{k-1})\right]
 (P_{t_i^{(k)}}-P_{t_{i-1}^{(k)}})X(t)\varOmega\\=&(\sigma_{\theta''}^r
 -\sigma_{\theta'}^r)X(t)\varOmega,
\end{align*}
and thus
\[
 \|(S_{\theta''}^r-S_{\theta'}^r)\varOmega\|_2^2=
 \sum_{k=1}^m\sum_{i=1}^{l_k}\|(P_{t_i^{(k)}}-P_{t_{i-1}^{(k)}})
 \big[f(t_{i-1}^{(k)})-f(t_{k-1})\big]
 X(t)\varOmega\|_{\mathcal{H}}^2
\]
by the orthogonality of $P_{t_i^{(k)}}-P_{t_{i-1}^{(k)}}\text{ and
}P_{t_j^{(k)}}-P_{t_{j-1}^{(k)}}\text{ for }i\ne j$. Furthermore
\begin{align*}
 &\|(P_{t_i^{(k)}}-P_{t_{i-1}^{(k)}})\big[f(t_{i-1}^{(k)})-f(t_{k-1})\big]
 X(t)\varOmega\|_{\mathcal{H}}^2\\=&\langle(P_{t_i^{(k)}}-P_{t_{i-1}^{(k)}})
 |f(t_{i-1}^{(k)})-f(t_{k-1})|^2(P_{t_i^{(k)}}-P_{t_{i-1}^{(k)}})
 X(t)\varOmega,X(t)\varOmega\rangle,
\end{align*}
and since $|t_{i-1}^{(k)}-t_{k-1}|<\delta$, we have
\[
 |f(t_{i-1}^{(k)})-f(t_{k-1})|^2\leqslant
 \frac{\varepsilon^2}{4\|X(t)\varOmega\|_{\mathcal{H}}^2}\boldsymbol{1},
\]
giving
\[
 (P_{t_i^{(k)}}-P_{t_{i-1}^{(k)}})|f(t_{i-1}^{(k)})-f(t_{k-1})|^2
 (P_{t_i^{(k)}}-P_{t_{i-1}^{(k)}})\leqslant\frac{\varepsilon^2}
 {4\|X(t)\varOmega\|_{\mathcal{H}}^2}(P_{t_i^{(k)}}-P_{t_{i-1}^{(k)}}).
\]
This yields the estimate
\begin{align*}
 &\|(P_{t_i^{(k)}}-P_{t_{i-1}^{(k)}})\big[f(t_{i-1}^{(k)})-f(t_{k-1})\big]
 X(t)\varOmega\|_{\mathcal{H}}^2\\
 =&\langle(P_{t_i^{(k)}}-P_{t_{i-1}^{(k)}})
 |f(t_{i-1}^{(k)})-f(t_{k-1})|^2(P_{t_i^{(k)}}-P_{t_{i-1}^{(k)}})
 X(t)\varOmega,X(t)\varOmega\rangle\\ \leqslant
 &\frac{\varepsilon^2}{4\|X(t)\varOmega\|_{\mathcal{H}}^2}
 \|(P_{t_i^{(k)}}-P_{t_{i-1}^{(k)}})X(t)\varOmega\|_{\mathcal{H}}^2.
\end{align*}
Consequently, we obtain
\begin{equation}\label{e1}
 \begin{aligned}
 \|(S_{\theta''}^r-S_{\theta'}^r)\varOmega\|_{\mathcal{H}}^2
 &=\sum_{k=1}^m\sum_{i=1}^{l_k}\|(P_{t_i^{(k)}}-P_{t_{i-1}^{(k)}})
 \big[f(t_{i-1}^{(k)})-f(t_{k-1})\big]X(t)\varOmega\|_{\mathcal{H}}^2\\
 &\leqslant\frac{\varepsilon^2}{4\|X(t)\varOmega\|_{\mathcal{H}}^2}
 \sum_{k=1}^m\sum_{i=1}^{l_k}\|(P_{t_i^{(k)}}-P_{t_{i-1}^{(k)}})
 X(t)\varOmega\|_{\mathcal{H}}^2\\
 &=\frac{\varepsilon^2}{4\|X(t)\varOmega\|_{\mathcal{H}}^2}
 \Big\|\sum_{k=1}^m\sum_{i=1}^{l_k}(P_{t_i^{(k)}}-P_{t_{i-1}^{(k)}})
 X(t)\varOmega\Big\|_{\mathcal{H}}^2\\
 &=\frac{\varepsilon^2}{4\|X(t)\varOmega\|_{\mathcal{H}}^2}
 \|(P_t-P_0)X(t)\varOmega\|_{\mathcal{H}}^2\\
 &\leqslant\frac{\varepsilon^2}{4\|X(t)\varOmega\|_{\mathcal{H}}^2}
 \|X(t)\varOmega\|_{\mathcal{H}}^2=\frac{\varepsilon^2}{4}.
 \end{aligned}
\end{equation}
Let now $\theta_1\text{ and }\theta_2$ be arbitrary partitions of
$[0,t]$ such that \linebreak
$\|\theta_1\|<\delta,\,\|\theta_2\|<\delta$, and let
$\theta''=\theta_1\cup\theta_2$. Then we have by \eqref{e1}
\[
 \|(S_{\theta''}^r-S_{\theta_1}^r)\varOmega\|_{\mathcal{H}}<\frac{\varepsilon}
 {2}\quad\text{and}\quad\|(S_{\theta''}^r-S_{\theta_2}^r)
 \varOmega\|_{\mathcal{H}}<\frac{\varepsilon}{2},
\]
so
\[
 \|(S_{\theta_1}^r-S_{\theta_2}^r)\varOmega\|_{\mathcal{H}}<\varepsilon,
\]
showing that the net $\{S_{\theta}^r\varOmega\}$ is Cauchy. Thus
there exists $\lim_{\|\theta\|\to 0}S_{\theta}^r\varOmega$ and the
rest of the proof is the same as that of Theorem \ref{mi}.
\end{proof}
\section{Integrals as operators on $\mathcal{H}$ --- the tracial case}
Let us now assume that $\omega$ is a normal tracial state. Recall
that the Lebesgue space $L^2(\mathcal{A},\omega)$ is formally
defined as the completion of $\mathcal{A}$ with respect to the
norm
\[
 \|x\|_2=\left[\omega(x^*x)\right]^{1/2}=\|x\varOmega\|_{\mathcal{H}},
\]
and may be realized as a space of densely defined closed operators
affiliated with $\mathcal{A}$ such that $\varOmega$ belongs to
their domains.

For the von Neumann subalgebras $\mathcal{A}_t$ the normal
$\omega$-invariant conditional expectations $\mathbb{E}_t
\colon\mathcal{A}\to\mathcal{A}_t$ extend to orthogonal
projections (denoted by the same letter) from
$L^2(\mathcal{A},\omega)\text{ onto }L^2(\mathcal{A}_t,\omega)$.
If we define an operator $P_t\text{ on }\mathcal{H}$ by
\begin{equation}\label{e2}
 P_t(X\varOmega)=(\mathbb{E}_tX)\varOmega,\qquad X\in
 L^2(\mathcal{A},\omega),
\end{equation}
then $P_t$ is an orthogonal projection from $\mathcal{A}'_t$. (The
definition above is the same as that given by \eqref{1} for
$X\in\mathcal{A}$.) For any $x\in\mathcal{A},\,A\in
L^2(\mathcal{A},\omega)$ the operators $xA\text{ and }Ax$ belong to
$L^2(\mathcal{A},\omega)$, consequently we may again consider the
integral sums
\[
 \sum_{k=1}^m f(t_{k-1})\Delta X(t_k),\qquad\sum_{k=1}^m \Delta
 X(t_k)f(t_{k-1}),
\]
where $f\colon[0,\infty)\to\mathcal{A},\: X\colon[0,\infty)\to
L^2(\mathcal{A},\omega)$, as operators on $\mathcal{H}$. We shall
use the notation
\[
 S_{\theta}^r(f,X)=\sum_{k=1}^m f(t_{k-1})\Delta X(t_k),\qquad
 S_{\theta}^l(f,X)=\sum_{k=1}^m \Delta X(t_k)f(t_{k-1}),
\]
for the integral sums, to indicate their dependence on $f\text{ and
}X$. Then since $S_{\theta}^r(f,X)\text{ and }S_{\theta}^l(f,X)$
are affiliated with $\mathcal{A}$, we have an explicit description
of their actions on $\mathcal{A}'\varOmega$ as
\[
 S_{\theta}^r(f,X)(x'\varOmega)=x'S_{\theta}^r(f,X)\varOmega,\qquad
 S_{\theta}^l(f,X)(x'\varOmega)=x'S_{\theta}^l(f,X)\varOmega.
\]
\begin{theorem}\label{ti}
Let $(X(t):t\geqslant 0)$ be a martingale in
$L^2(\mathcal{A},\omega)$ and let
$f\colon[0,\infty)\to\mathcal{A}$ be either monotone or norm
continuous. Then for each $t>0$ there exist integrals $\int_0^t
f\,dX\text{ and }\int_0^t dX\,f$ as elements of
$L^2(\mathcal{A},\omega)$. Moreover, the
$L^2(\mathcal{A},\omega)$-processes
$(Y(t):t\geqslant0)$,\newline$(Z(t):t\geqslant0)$ defined by
\[
 Y(t)=\int_0^t dX\,f, \qquad Z(t)=\int_0^t f\,dX
\]
are martingales.
\end{theorem}
\begin{proof} The existence of the integral $\int_0^t f\,dX$ is
proved exactly as in Theorems \ref{mi} and \ref{ci} upon observing
that according to formula \eqref{e2} we have
\[
 \left[X(t_k)-X(t_{k-1})\right]\varOmega=\left[\mathbb{E}_{t_k}X(t)
 -\mathbb{E}_{t_{k-1}}X(t)\right]\varOmega=(P_{t_k}-P_{t_{k-1}})X(t)\varOmega.
\]
It follows that the net $\{S_{\theta}^r(f,X)\varOmega\}$ is
Cauchy, and since
\[
 \|S_{\theta}^r(f,X)\|_{\mathcal{H}}=\|S_{\theta}^r(f,X)\|_2,
\]
$\{S_{\theta}^r(f,X)\}$ converges in $\|\cdot\|_2$-norm, and
thus its limit is an element of $L^2(\mathcal{A},\omega)$.

For $S_{\theta}^l(f,X)$ we have
\[
 S_{\theta}^l(f,X)=\left[S_{\theta}^r(f^*,X^*)\right]^*;
\]
so we obtain
\[
 \|S_{\theta}^l(f,X)\|_2=\|\left[S_{\theta}^r(f^*,X^*)\right]^*\|_2
 =\|S_{\theta}^r(f^*,X^*)\|_2,
\]
because $\varOmega$ is tracial. Since $f^*$ satisfies the same
assumptions as $f$, and $(X(t)^*:t\geqslant0)$ is also an
$L^2(\mathcal{A},\omega)$--martingale, we obtain the convergence of
$\{S_{\theta}^l(f,X)\}$ in $\|\cdot\|_2$-norm.

Now we shall show that $(Y(t):t\geqslant0)$ is a martingale. Fix
$t>0$ and take an arbitrary $s<t$. We have
\[
 \int_0^t dX\,f = \lim_{\|\theta\|\to 0}S_{\theta}^l.
\]
We may assume that $s$ is one of the points of each partition
\linebreak $\theta=\{0=t_0<t_1<\cdots<t_m=t\}$, say $s=t_k$. Then we
have
\begin{align*}
 &\mathbb{E}_s S_{\theta}^l=\mathbb{E}_s(\sum_{i=1}^k
 [X(t_i)-X(t_{i-1})]f(t_{i-1})+\sum_{i=k+1}^m
 [X(t_i)-X(t_{i-1})]f(t_{i-1}))\\
 &=\sum_{i=1}^k\mathbb{E}_s[X(t_i)-X(t_{i-1})]f(t_{i-1})+
 \sum_{i=k+1}^m\mathbb{E}_s[X(t_i)-X(t_{i-1})]f(t_{i-1}).
\end{align*}
For $i\leqslant k$ we have $t_i\leqslant s$, and thus
\[
 \mathbb{E}_s[X(t_i)-X(t_{i-1})]f(t_{i-1})=[X(t_i)-X(t_{i-1})]f(t_{i-1}),
\]
while for $i>k$ we have $t_{i-1}\geqslant s$, and thus
\begin{align*}
 &\mathbb{E}_s[X(t_i)-X(t_{i-1})]f(t_{i-1})=
 \mathbb{E}_s\mathbb{E}_{t_{i-1}}[X(t_i)-X(t_{i-1})]f(t_{i-1})\\
 =&\mathbb{E}_s(\mathbb{E}_{t_{i-1}}[X(t_i)-X(t_{i-1})])f(t_{i-1})=0
\end{align*}
by the martingale property. Consequently,
\begin{equation}\label{e3}
 \mathbb{E}_s
 S_{\theta}^l=\sum_{i=1}^k[X(t_i)-X(t_{i-1})]f(t_{i-1}).
\end{equation}
But the sum on the right hand side of \eqref{e3} is an integral
sum for the integral $\int_0^s dX\,f$, and passing to the limit in
\eqref{e3} yields
\[
 \mathbb{E}_s\int_0^t dX\,f = \int_0^s dX\,f,
\]
which shows that $(Y(t))$ is a martingale. Analogously
for\linebreak $(Z(t):t\geqslant0)$.
\end{proof}
\begin{remark}
1. The formulas
\[
 \bigg(\int_0^t f\,dX\bigg)(x'\varOmega)=
 x'\,\lim_{\theta}S_{\theta}^r\varOmega,\qquad
 \bigg(\int_0^t dX\,f\bigg)(x'\varOmega)=
 x'\,\lim_{\theta}S_{\theta}^l\varOmega
\]
describe explicitly the actions of the integrals on
$\mathcal{A}'\varOmega$.
\par 2. Let us notice that an attempt to define Lebesgue type
integrals above, e.g., along the lines of \cite{4} or \cite{10},
would be successful only in the simple case of norm-continuous $f$,
and even then under the additional assumption of left- or
right-continuity in $\|\cdot\|_2$-norm of the martingale $(X(t))$.
The reason for this is that this type of integral is defined for
$\mu$--measurable functions $f\in L^2([0,\infty),\mu,\mathcal{A})$,
where $\mu$ is a measure defined by
\[
 \mu([a,b))\text{ or }\mu((a,b])=\omega(|X(b)-X(a)|^2)
 =\omega(|X(b)|^2)-\omega(|X(a)|^2).
\]
The continuity of the martingale allows then the extension of $\mu$
from the intervals to the Borel sets. For the case of monotone $f$
the failure of the definition of the integral as Lebesgue type is
most strikingly seen when the function $f$ is increasing
projection-valued. To be more concrete, assume that $e$ is a spectral
measure with support $[0,\infty)$, and put $f(t)=e([0,t])$. Then for
$s < t$, $f(t) - f(s)$ is a non-zero projection, so $\|f(t) -
f(s)\|=1$, which means that $f$ is not norm continuous either from
the left or from the right at any point, while a $\mu$--measurable
function must be norm-continuous on some compact set.
\end{remark}

\end{document}